\documentclass{amsart}
\usepackage{amsfonts}
\usepackage{graphicx}
\usepackage{amscd}
\usepackage{amsmath}
\usepackage{amssymb}
\usepackage{stmaryrd} 
\usepackage{enumitem}

\usepackage{tikz-cd}
\usepackage[pdf]{pstricks}
\usepackage{chemfig}
\usepackage{chemmacros}
\usetikzlibrary{cd}

\makeatletter
\@namedef{subjclassname@2010}{%
  \textup{2010} Mathematics Subject Classification}
\makeatother

\theoremstyle{plain}

\newtheorem{corollary}{\bf Corollary}

\newtheorem{lemma}{\bf Lemma}

\newtheorem{proposition}{\bf Proposition}

\newtheorem{theorem}{\bf Theorem}

\theoremstyle{definition}

\numberwithin{equation}{section}

\title[Einstein hypersurfaces]{O\MakeLowercase{n} E\MakeLowercase{instein hypersurfaces of} $I\times_{f}\mathbb{Q}^{n}(c)$}

\author{V. Borges}
\author{A. Da Silva}

\address[V. Borges]{Faculdade de Matem\'{a}tica, Universidade Federal do Par\'{a}\\
	66075-110 Be\-l\'{e}m, Par\'{a}, Brazil.}
\email{valterborges@ufpa.br}

\address[A. da Silva]{Faculdade de Matem\'{a}tica, Universidade Federal do Par\'{a}\\
	66075-110 Be\-l\'{e}m, Par\'{a}, Brazil.}
\email{adamsilva@ufpa.br}


\begin{document}

\newcommand{\spacing}[1]{\renewcommand{\baselinestretch}{#1}\large\normalsize}
\spacing{1.2}

\begin{abstract}
In this paper, we investigate Einstein hypersurfaces of the warped product $I\times_{f}\mathbb{Q}^{n}(c)$, where $\mathbb{Q}^{n}(c)$ is a space form of curvature $c$. We prove that $M$ has at most three distinct principal curvatures and that it is locally a multiply warped product with at most two fibers. We also show that exactly one or two principal curvatures on an open set imply constant sectional curvature on that set. For exactly three distinct principal curvatures this is no longer true, and we classify such hypersurfaces provided it does not have constant sectional curvature and a certain principal curvature vanishes identically.
\end{abstract}

\maketitle

\section{Introduction and Main Results}\label{intro}


The investigation of Einstein hypersurfaces of space forms $\mathbb{Q}^{n+1}(c)$ goes back to the mid $30$'s \cite{fialkow,thomas} and a complete local classification was obtained in $1968$, by Ryan \cite{Ryan}. His proof goes through showing that these hypersurfaces have at most two distinct principal curvatures at each point. Using the integral manifolds of the distributions corresponding to the principal curvatures, he shows that either $M$ has constant sectional curvature, which is not smaller than $c$, or $c>0$ and the Einstein metric splits off as certain Riemannian products between spheres.

A complete classification is also known when the ambient space is a complex or quaternionic space form of non zero curvature. In this direction, the first result came out in 1982, when Cecil and Ryan \cite{cery} showed that there is no real Einstein hypersurface in the complex projective space $\mathbb{CP}^n$. Three years later, Montiel \cite{montiel} showed the same result in the complex hyperbolic space $\mathbb{CH}^n$. Martinéz and Pérez \cite{mape}, in turn, showed in 1985 that certain geodesic hyperspheres are the only real Einstein hypersurfaces of the quaternionic projective space $\mathbb{QP}^n$. The quaternionic hyperbolic space $\mathbb{QH}^n$, on the other hand, admits no real Einstein hypersurface, as proved in 1996 by Ortega and Peréz \cite{orper}.

Einstein hypersurfaces of the cylinders $\mathbb{R}\times\mathbb{S}^{n}$ and $\mathbb{R}\times\mathbb{H}^{n}$ have also recently been classified \cite{SantosPN}. Here the classical theory of isoparametric hypersurfaces of space forms, developed by Cartan \cite{cartan}, plays an important role. Another important point in this classification is the fact that the orthogonal projection of $\partial_{t}$ on the Einstein hypersurface is a principal direction, where $\partial_{t}$ is the field spanning the factor $\mathbb{R}$. Their main result asserts that Einstein hypersurfaces in these ambient have constant sectional curvature, which allowed the use of a previous classification.

In this paper we investigate Einstein manifolds $M^n$, satisfying $Ric=\rho g$, which are hypersurfaces of the warped product $I\times_{f}\mathbb{Q}^{n}(c)$. As three dimensional Einstein manifolds have constant sectional curvature, our main results are concerned with the case where $n\geq4$. Here, $f:I\rightarrow\mathbb{R}$ is a positive smooth function, $\mathbb{Q}^{n}(c)$ is space form of constant curvature $c\in\{-1,0,1\}$ and $I$ is an interval. This encloses a large variety of ambient spaces. For instance, space forms and cylinders can be obtained in this way by choosing suitable warping functions. However, we notice that our arguments do not apply to the case where the ambient $I\times_{f}\mathbb{Q}^{n}(c)$ has constant sectional curvature.

In order to state the main results of this paper, let us still denote by $\partial_{t}$ the direction spanning $I$. Consider the orthogonal decomposition $\partial_{t}=T+\theta\mathcal{N}$ on points of $M$, where $\mathcal{N}$ is the unit normal field of $M$. Both $T$ and $\theta$ play an important role in the investigation of hypersurfaces of $I\times_{f}\mathbb{Q}^{n}(c)$. For instance, when $T$ vanishes identically, $M$ is isometric to $\{t\}\times\mathbb{Q}^{n}(c)$ with the induced metric (see \cite{benoa,Ortega}), which has constant sectional curvature. In this case we say that $M$ is a {\it slice}. Our main theorem is the following one.

\begin{theorem}\label{intM}
Let $(M^{n},g)$, $n\geq4$, be an Einstein hypersurface of $I\times_{f}\mathbb{Q}^{n}(c)$ and suppose that $I\times_{f}\mathbb{Q}^{n}(c)$ does not have constant sectional curvature. Let $U\subset M$ be an open set where $T$ does not vanish. Then
\begin{enumerate}
	\item\label{itone} $M$ has at least $2$ and at most $3$ distinct principal curvatures on $U$;
	\item\label{ittwo} $T$ is a principal direction of $M$ on $U$ with multiplicity $1$;
	\item if $M$ has exactly $3$ distinct principal curvatures on $U$, say $\lambda_{n}$, $\lambda_{1}$ and $\lambda_{2}$, then their multiplicities $1$, $p_{1}$ and $p_{2}$, respectively, are constant, and $U$ is locally isometric to $J\times N_{1}^{p_{1}}\times N_{2}^{p_{2}}$, with metric
	\begin{equation*}\label{doubwarpint}
		g=ds^2+\varphi_{1}(s)^2g_{1}+\varphi_{2}(s)^2g_{2},
	\end{equation*}
	where $J\subset\mathbb{R}$ is an interval, $(N_{i}^{p_{i}},g_{i})$ is a space form and $\varphi_{i}:J\rightarrow\mathbb{R}$ is a positive smooth function, $i\in\{1,2\}$.
\end{enumerate}
\end{theorem}

A first consequence of item $(\ref{itone})$ is that if $M$ is totally umbilical, then $T$ vanishes. Consequently,

\begin{corollary}\label{corollary}
	Let $(M^{n},g)$, $n\geq4$, be an Einstein hypersurface of $I\times_{f}\mathbb{Q}^{n}(c)$ and suppose that $I\times_{f}\mathbb{Q}^{n}(c)$ does not have constant sectional curvature. If $M$ is totally umbilical, then it is a slice.
\end{corollary}

In order to prove Theorem \ref{intM}, we show that the distributions defined by the principal curvatures are involutive and manage to show that the metric of $M$ splits as a multiply warped product (see Lemma \ref{integrability} and Proposition \ref{prop1}). We remark that $T$ being a principal direction is essential in our proof. It is exactly the case if the ambient space does not have constant sectional curvature (see Proposition \ref{Tdiag}). Our tools are the fundamental equations describing a hypersurface of $I\times_{f}\mathbb{Q}^{n}(c)$, recently obtained in \cite{Ortega} (\cite{benoa}, for $f$ constant), and Frobenius theorem. 

When $M$ has only two distinct principal curvatures in $U\subset M$, we are able to show that $U$ is locally conformally flat. Consequently we have the following.

\begin{theorem}\label{rigidity1}
	Let $(M^{n},g)$, $n\geq4$, be an Einstein hypersurface of $I\times_{f}\mathbb{Q}^{n}(c)$ and suppose that $I\times_{f}\mathbb{Q}^{n}(c)$ does not have constant sectional curvature. Let $U\subset M$ be an open set where $T$ does not vanish and where $M$ has exactly two principal curvatures. Then, $M$ has constant sectional curvature on $U$.
\end{theorem}
When $M$ has three distinct principal curvatures on $U$, Theorem \ref{intM} says that $U$ is locally isometric to a multiply warped product with one dimensional base and two fibers of constant sectional curvatures, say, $k_{1}$ and $k_{2}$, and warping functions $\varphi_{1}$ and $\varphi_{2}$. The curvature of such spaces can be expressed in terms of $k_{1}$, $k_{2}$, $\varphi_{1}$, $\varphi_{2}$ and their derivatives. Such expressions can be found, for example, in \cite{wpmandlcfs}. When the principal curvature corresponding to $T$, denoted by $\lambda_{n}$, vanishes, these formulas for the curvature of $M$ together with Gauss equation yield the following result.

\begin{theorem}\label{thmlambda_{n}=0}
	Let $(M^{n},g)$, $n\geq4$, be an Einstein hypersurface of $I\times_{f}\mathbb{Q}^{n}(c)$ and suppose that $I\times_{f}\mathbb{Q}^{n}(c)$ does not have constant sectional curvature. Let $U\subset M$ be an open set where $T$ does not vanish and where $\lambda_{n}$ vanishes. Then, $M$ does not have constant sectional curvature in $U$ if and only if the following are true
	\begin{enumerate}
		\item $c=1$, $p_{1}\geq2$, $p_{2}\geq2$, $n\geq5$, $\theta\equiv0$ and $f$ is a solution of
		\begin{equation*}
			\left(\frac{df}{dt}\right)^2+\frac{\rho}{n-1}f^2=\frac{n-3}{n-2};
		\end{equation*}
		\item there are positive constants $k_{1}$ and $k_{2}$, so that $U$ is locally isometric to
		\begin{equation}\label{example1}
			I\times_{\varphi_{1}}\mathbb{S}^{p_{1}}\left(\frac{1}{\sqrt{k_{1}}}\right)\times_{\varphi_{2}}\mathbb{S}^{p_{2}}\left(\frac{1}{\sqrt{k_{2}}}\right),
		\end{equation}
	\end{enumerate}
	where $\varphi_{i}$ and the principal curvature $\lambda_{i}$, whose multiplicity is $p_{i}$, are given by 
	\begin{align*}
		\varphi_{i}=\sqrt{\frac{(p_{i}-1)k_{i}}{n-3}}f,\ \ \lambda_{i}=(-1)^{i}\sqrt{\frac{p_{j}-1}{p_{i}-1}}\frac{1}{f},
	\end{align*}
where $i\in\{1,2\}$ and $j=3-i$.
\end{theorem}

We remark that three distinct principal curvatures do not occur neither in cylinders \cite{SantosPN} nor in space forms \cite{Ryan}. This is in contrast with example $(\ref{example1})$.

The paper is organized as follows. In Section 2 we state the structure equations for the existence of an immersion of $M^n$ in $I\times_{f}\mathbb{Q}^n(c)$. Our first result, which could be useful elsewhere, concerns hypersurfaces for which $T$ is a principal direction. In Section 3 we investigate extrinsic and intrinsic properties of $M$. We prove that $M$ has at most three distinct principal curvatures and that it is locally isometric to a multiply warped product with at most two fibers (see \cite{wpmandlcfs} for definitions). In Section 4, we prove the main theorems. We also rewrite the proof of the main result of \cite{SantosPN}. See Theorem \ref{rigidity2}.

\section{Preliminaries}\label{Definitions}

The goal of this section is to fix notation and to set the main tools used along this paper.
\subsection{Notations}
Consider the simply connected, complete, $n$-dimensional space form $\mathbb{Q}^n (c)$ having constant curvature $c\in \{-1, 0, 1\}$ and a smooth real function $f: I\subseteq \mathbb{R} \rightarrow \mathbb{R}^{+}$. The warped product $I \times_{f} \mathbb{Q}^{n}(c)$ is the product manifold $I\times \mathbb{Q}^n (c)$ endowed with metric $\langle, \rangle= dt^2+ f(t)^2 g$, where $g$ denotes the metric of $\mathbb{Q}^{n}(c)$.

Let $M^n$ be an oriented hypersurface of $I \times_{f} \mathbb{Q}^{n}(c)$ with Levi-Civita connection $\nabla$. We set our convention for the curvature operator $\mathcal{R}$ as 
$$\mathcal{R}(X, Y)Z =  \nabla_{X}\nabla_{Y}Z - \nabla_{Y}\nabla_{X}Z- \nabla_{[X, Y]}Z.$$

Let $\mathcal{N}$ be a normal unit vector field to $M$ in $I\times_{f}\mathbb{Q}^n (c)$ and let $A$ be the shape operator associated with $\mathcal{N}$. The mean curvature of $M^n$ is defined by
$$H = \dfrac{1}{n}Tr(A).$$

If $\partial_{t}$ denotes the field tangent to $I$ in $I\times \mathbb{Q}^n (c)$, then it can be decomposed at points of $M$ as
$$\partial_{t}= T+\theta \mathcal{N},$$
where $T$ is tangent to $M$ and $\theta = \langle \partial_{t}, \mathcal{N}\rangle $. The function $\theta$ is called {\it angle function}.

\subsection{Structure equations}
In this subsection we state the structure equations to obtain a Riemannian manifold $(M^n,g)$ as a hypersurface of $I\times_{f}\mathbb{Q}^n(c)$. This result was first obtained for $f$ constant in \cite{benoa}, and then for general $f$ in \cite{Ortega}.

Let $(M^{n},\left\langle ,\right\rangle)$ be a Riemannian manifold with Levi-Civita connection $\nabla$ and $(0,4)$-curvature tensor $\mathcal{R}$. We say that $M$ satisfies the {\it structure conditions} if there are: a smooth $(1,1)$-tensor $A:\mathfrak{X}(M)\rightarrow\mathfrak{X}(M)$, a constant $c\in\{-1,0,1\}$, an interval $I\subset\mathbb{R}$ and smooth functions $f:I\rightarrow(0,+\infty)$, $\theta:M\rightarrow\mathbb{R}$ and $\pi:M\rightarrow I$, with $T=\nabla\pi$, satisfying

\begin{enumerate}[label=(\Alph*)]
	\item\label{selfadj} $A$ is self adjoint with respect to $\left\langle,\right\rangle$;
	\item\label{decomp} $|T|^2+\theta^2=1$;
	\item\label{prehess} $\nabla_{X}T=\frac{f'}{f}(X-\left\langle X,T\right\rangle T)+\theta AX,\ \forall X\in\mathfrak{X}(M)$;
	\item\label{anglefunction} $X(\theta)=-\left\langle AT,X\right\rangle-\frac{f'}{f}\theta\left\langle X,T\right\rangle,\ \forall X\in\mathfrak{X}(M)$;
	\item\label{Codeq} For any $X,\ Y\in\mathfrak{X}(M)$,
	\begin{align*}
	(\nabla_{X}A)Y-(\nabla_{Y}A)X=\theta b\big(\left\langle T,X\right\rangle Y-\left\langle T,Y\right\rangle X\big);
	\end{align*}
	\item\label{Gausseq} For any $X, Y, Z, W\in TM$,
	\begin{align*} \mathcal{R}(X, Y, Z, W)&= a \big(\langle X, Z\rangle \langle Y, W\rangle  -  \langle Y, Z \rangle \langle X, W\rangle \big)\\
	&+b \big( \langle X, Z\rangle \langle Y, T\rangle \langle W, T\rangle
	- \langle Y, Z\rangle \langle X, T\rangle \langle W, T\rangle \\
	& - \langle X, W\rangle \langle Y, T\rangle \langle Z, T\rangle  + \langle Y, W\rangle \langle X, T\rangle \langle Z, T\rangle \big)\\
	& + \langle AY, Z\rangle  \langle AX, W\rangle  - \langle AY, W\rangle \langle AX, Z\rangle,
	\end{align*}
\end{enumerate}
where
\begin{align}\label{defint.a}
a=\frac{(f')^2-c}{f^2},\ \ \ \ \ \ b=\frac{f''}{f}-\frac{(f')^2}{f^2}+ \frac{c}{f^2}.
\end{align}

We observe that $b$ vanishes identically if and only if $I\times_{f}\mathbb{Q}^{n}(c)$ has constant sectional curvature (see Proposition 2 of \cite{Ortega}).

The result below is the main tool in investigating hypersurfaces in the warped product $I\times_{f}\mathbb{Q}^{n}(c)$. Its proof can be found in \cite{Ortega}.

\begin{theorem}[\cite{Ortega}]\label{tensorR}
	Let $(M^n,\left\langle,\right\rangle)$ be a Riemannian manifold satisfying the structure conditions. Then, for each point $p\in M$, there exists a neighborhood $U$ of $p$ in $M$, an isometric immersion $\chi:U\rightarrow I\times_{f}\mathbb{Q}^{n}(c)$ and a normal vector field $\mathcal{N}$ along $\chi$ such that
	\begin{enumerate}
		\item $\mathcal{N}$ has length one in $I\times_{f}\mathbb{Q}^{n}(c)$;
		\item $\pi_{1}\circ\chi=\pi$, where $\pi_{1}:I\times_{f}\mathbb{Q}^{n}(c)\rightarrow I$ is the projection in the first component;
		\item The shape operator associated with $\mathcal{N}$ is $A$;
		\item $\ref{Codeq}$ is the Codazzi equation and $\ref{Gausseq}$ is the Gauss equation;
		\item Along $\chi$ one has $\partial_{t}=T+\theta\mathcal{N}$.
	\end{enumerate}
	Conversely, if such immersion exists, then $M$ satisfies the structure equations.
\end{theorem}

By tracing \ref{Gausseq} we obtain as a consequence of Theorem \ref{tensorR} the following.

\begin{corollary}\label{Riccitensor}
The Ricci tensor $Ric$ of $M$ in $I\times_{f} \mathbb{Q}^{n}(c)$ is given by
\begin{align*}
Ric(X,Y)=&-((n-1)a+|T|^2 b)\langle X,Y\rangle-(n-2)b\langle X,T\rangle \langle Y,T\rangle\\&\ \ \ \ \ +nH\langle AX,Y\rangle-\langle AX,AY\rangle.
\end{align*}
\end{corollary}

\subsection{Hypersurfaces of $I\times_{f} \mathbb{Q}^{n}(c)$ for which $T$ is a principal direction}\label{T_eigen}

In this subsection we investigate hypersurfaces of the warped product $I\times_{f}\mathbb{Q}^{n}(c)$ which have $T\neq0$ as principal direction on an open set.

\begin{proposition}\label{impTeigen}
	Let $M^n$ be a hypersurface of $I\times_{f} \mathbb{Q}^{n}(c)$ and $U\subset M$ an open set. Assume that $T$ is an eigenvector of $A:\mathfrak{X}(M)\rightarrow\mathfrak{X}(M)$, the shape operator of M, satisfying $AT=\lambda_{n}T$ on $U$. The following is true on $U$:
	\begin{enumerate}
		\item\label{integofT} The integral curves of $\frac{T}{|T|}$ are geodesics;
		\item\label{gradlambn} $\lambda_{n}$ is smooth and $\nabla\lambda_{n}=\left(\nabla_{\frac{T}{|T|}}A\right)\left(\frac{T}{|T|}\right)=\left\langle\frac{T}{|T|},\nabla\lambda_{n}\right\rangle\frac{T}{|T|}$;
		\item\label{gradtheta} $\nabla\theta=-\left(\lambda_{n}+\frac{f'}{f}\theta\right)$T.
	\end{enumerate}
	In particular, $|T|^2$, $\theta$ and $\lambda_{n}$ are constant on each connected component of regular levels of $\pi$ contained in $U$.
\end{proposition}
\begin{proof}
	Using \ref{decomp} and \ref{prehess} we have
	\begin{align*}
	\nabla_{T}T=\frac{f'}{f}(T-|T|^2 T)+\theta AT=\theta\left(\lambda_{n}+\frac{f'}{f}\theta\right)T,
	\end{align*}
	what implies that the integral curves of $\frac{T}{|T|}$ are geodesics, proving $(\ref{integofT})$.
	
	In order to prove $(\ref{gradlambn})$ we start noting that since $\lambda_{n}$ can be written as
	\begin{align*}
	\lambda_{n}=\frac{\left\langle AT,T\right\rangle}{|T|^2},
	\end{align*}
	it is smooth. To prove the first equality of $(\ref{gradlambn})$ we start differentiating the identity above with respect to $X\in\mathfrak{X}(M)$ to obtain
	\begin{align*}
	|T|^4\left\langle\nabla\lambda_{n},X\right\rangle&=(\left\langle \nabla_{X}(AT),T\right\rangle+\left\langle AT,\nabla_{X}T\right\rangle)|T|^2-2\left\langle\nabla_{X}T,T\right\rangle\left\langle AT,T\right\rangle\\
	&=(\left\langle \nabla_{X}(AT),T\right\rangle-\lambda_{n}\left\langle T,\nabla_{X}T\right\rangle)|T|^2,
	\end{align*}
	which gives
	\begin{align*}
	|T|^2\left\langle\nabla\lambda_{n},X\right\rangle=\left\langle \nabla_{X}(AT),T\right\rangle-\left\langle AT,\nabla_{X}T\right\rangle.
	\end{align*}
	On the other hand, it follows from \ref{Codeq} that
	\begin{align*}
	\left\langle \nabla_{X}(AT),T\right\rangle&=\left\langle(\nabla_{X}A)T+A(\nabla_{X}T),T\right\rangle\\
	&=\left\langle(\nabla_{T}A)X,T\right\rangle+\theta b\left\langle \left\langle T,X\right\rangle T-|T|^2 X,T\right\rangle+\left\langle \nabla_{X}T,AT\right\rangle\\
	&=\left\langle(\nabla_{T}A)T,X\right\rangle+\left\langle \nabla_{X}T,AT\right\rangle,
	\end{align*}
	where in the last line we have used the identity $\left\langle(\nabla_{Z}A)X,Y\right\rangle=\left\langle(\nabla_{Z}A)Y,X\right\rangle$, $\forall X,Y,Z\in\mathfrak{X}(M)$. Consequently, $|T|^2\left\langle\nabla\lambda_{n},X\right\rangle=\left\langle(\nabla_{T}A)T,X\right\rangle,$ as we claimed.
	
	In order to finish the proof of $(\ref{gradlambn})$ notice that since the integral curves of $\frac{T}{|T|}$ are geodesics, we have
	\begin{equation*}
		\left(\nabla_{\frac{T}{|T|}}A\right)\left(\frac{T}{|T|}\right)=\nabla_{\frac{T}{|T|}}\left(A\left(\frac{T}{|T|}\right)\right)
		=\nabla_{\frac{T}{|T|}}\left(\lambda_{n}\frac{T}{|T|}\right)
		=\frac{\left\langle T,\nabla\lambda_{n}\right\rangle}{|T|^2}T,
	\end{equation*}
	which is the second equality of $(\ref{gradlambn})$.
	
	Now we prove $(\ref{gradtheta})$. Using identity \ref{anglefunction} we have
	\begin{align*}
	\left\langle\nabla\theta,X\right\rangle=X(\theta)=-\left\langle AT+\frac{f'}{f}\theta T,X\right\rangle=\left\langle-\left(\lambda_{n}+\frac{f'}{f}\theta\right)T,X\right\rangle,
	\end{align*}
	for any $X\in\mathfrak{X}(M)$, which implies $(\ref{gradtheta})$.
\end{proof}

\section{Einstein hypersurfaces of $I\times_{f} \mathbb{Q}^{n}(c)$}

In this section we investigate Einstein hypersurfaces of $I\times_{f} \mathbb{Q}^{n}(c)$ assuming $b\neq0$. This is equivalent to saying that $I\times_{f} \mathbb{Q}^{n}(c)$ does not have constant sectional curvature (see Proposition 2 of \cite{Ortega}). A complete local classification of Einstein hypersurfaces in spaces of constant sectional curvature can be found in \cite{Ryan}. Recall that as three dimensional Einstein manifolds have constant sectional curvature, we will assume from now on that $n\geq4$.

\subsection{Extrinsic structure of $M$}
In this subsection we investigate properties of the principal curvatures of the immersion on an open set $U\subset M$ where $T$ and $b$ do not vanish. We start with the following proposition which, in particular, allows us to use Proposition \ref{impTeigen}.

\begin{proposition}\label{Tdiag}
	Let $M^n$, $n\ge 4$, be an Einstein hypersurface of $I\times_{f} \mathbb{Q}^{n}(c)$. If $T\neq 0$ and $b\neq 0$ at $x_{0}\in M^n$, then $T$ is an eigenvector of the shape operator $A$ at $x_{0}$. Furthermore, $A$ has at least $2$ and most $3$ distinct principal curvatures at $x_{0}$, say, $\lambda_{1}, \lambda_{2}$ and $\lambda_{n}$, which satisfy
	\begin{align}
		&\lambda_{i}^2-nH\lambda_{i}+\rho+(n-1)a+|T|^2 b=0,\ i\in\{1,2\},\label{eq1}\\
 		&\lambda_{n}^2-nH\lambda_{n}+\rho+(n-1)(a+|T|^2 b)=0.\label{eq2}
	\end{align}
	In particular, the multiplicity of $\lambda_{n}$ as an eigenvalue of $A$ at $x_{0}$ is $1$.
\end{proposition}

\begin{proof}
	Let $\{e_{i}\}_{i=1}^{n}$ be an orthonormal basis of principal directions at $x_{0}\in M$, with $Ae_{i} = \lambda_{i}e_{i}$. If we write $T= \sum_{i=1}^{n} t_{i}e_{i}$, Corollary \ref{Riccitensor} gives the equality
	\begin{align*}
	R_{ij}&=-(n-1)a\delta_{ij} - b\Big(|T|^2 \delta_{ij}+ (n-2)t_{i}t_{j}\Big) + nH\lambda_{i}\delta_{ij} - \lambda_{i}\lambda_{j}\delta_{ij}.
	\end{align*}
	Since $M^n$ is an Einstein hypersurface satisfying $Ric=\rho g$, we must have
	\begin{align}\label{A7}
	\left[\lambda_{i}\lambda_{j}-nH\lambda_{i}+\rho+(n-1)a+|T|^2 b\right]\delta_{ij}=-b(n-2)t_{i}t_{j}.
	\end{align}
	Hence, we conclude that $t_{i}t_{j}=0$ for all $i\neq j$. As $T\neq 0$ at $x_{0}$, there is only one coefficient $t_{k}\neq 0$. Without loss of generality, assume that $k=n$. Thus, $AT=\lambda_{n}T$ at $x_{0}$. Then, it follows from $(\ref{A7})$ that $(\ref{eq1})$ and $(\ref{eq2})$ hold. Now, since all $\lambda_{i},\ i\in\{1,\ldots,n-1\},$ satisfy the second degree equation $(\ref{eq1})$, there are at most two of them, say, $\lambda_{1}$ and $\lambda_{2}$. Since $n>3$, $b\neq0$ and $T\neq0$, $(\ref{eq1})$ and $(\ref{eq2})$ imply that $\lambda_{1}\neq\lambda_{n}$ and $\lambda_{2}\neq\lambda_{n}$, which finishes the proof.
\end{proof}

As $T$ is a principal direction of $M$, Proposition \ref{impTeigen} says that $\lambda_{n}$ and $|T|^2$ are constant on the connected regular levels of $\pi$. The next proposition asserts that the same is true for $\lambda_{1}$ and $\lambda_{2}$. This is a consequence of $(\ref{eq1})$ and $(\ref{eq2})$.

\begin{proposition}\label{lambdannot0}
	Let $M^n$ be an Einstein hypersurface of $I\times_{f}\mathbb{Q}^{n}(c)$. Assume that $U\subset M$ is an open set where $T\neq 0$ and $b\neq 0$, and $\Sigma_{t}$ is a connected level of $\pi$ inside $U$. Then the principal curvatures of $M\subset I\times_{f}\mathbb{Q}^{n}(c)$ are constant on $\Sigma_{t}$. If there is, in addition, $x_{0}\in\Sigma_{t}$ so that $\lambda_{n}(x_{0})=0$, then:
	\begin{enumerate}
		\item If $b(x_{0})<0$, then $p_{1}=n-1$, $p_{2}=0$ on $\Sigma_{t}$.
		\item If $b(x_{0})>0$, then $p_{1}$ and $p_{2}$ are constant on $\Sigma_{t}$, $p_{1}\geq2$, $p_{2}\geq2$ and $n>4$.
	\end{enumerate}
	Here, $p_{i}=p_{i}(x)$ denotes the multiplicity of $\lambda_{i}(x)$, $i\in\{1,2\}$.
\end{proposition}
\begin{proof}
We start by observing that the functions $a$ and $b$ are constant on $\Sigma_{t}$ and since $T$ is an eigenvector of $A$, it follows from Proposition \ref{impTeigen} that $|T|^2$, $\theta$ and $\lambda_{n}$ are constant on $\Sigma_{t}$.

Suppose that $\lambda_{n}(x_{0})\neq0$ for some $x_{0}\in\Sigma_{t}$. Then we have $\lambda_{n}(x)=\lambda_{n}(x_{0})\neq0,\ \forall x\in \Sigma_{t}$. If $X\in\mathfrak{X}(\Sigma_{t})$, then $(\ref{eq2})$ gives $X(H)\lambda_{n}=0$ and then, $H$ is constant on $\Sigma_{t}$ which, together with $(\ref{eq1})$, implies that the remaining principal curvatures of $A$ are also constant on $\Sigma_{t}$.

Now assume that $\lambda_{n}(x_{0})=0$ for some $x_{0}\in\Sigma_{t}$. As $\lambda_{n}$ is constant on $\Sigma_{t}$, we have $\lambda_{n}(x)=0,\ \forall x\in \Sigma_{t}$. Consequently, for each $x\in \Sigma_{t}$ we have
	\begin{align}\label{interm1}
	nH=p_{1}\lambda_{1}+p_{2}\lambda_{2},
	\end{align}
	with $p_{1},p_{2}\in\{0,\ldots,n-1\}$ and $p_{1}+p_{2}=n-1$. Using $(\ref{eq2})$ we also have $\rho=-(n-1)(a+|T|^2 b),$	which by $(\ref{eq1})$ implies that $\lambda_{1}$ and $\lambda_{2}$ are solution of the equation for $y$
	\begin{align}\label{eq1'}
	y^2-nHy-(n-2)|T|^2 b=0.
	\end{align}
	Observe that the equation above implies that $(\lambda_{1}-\lambda_{2})(\lambda_{1}+\lambda_{2}-nH)=0$.
	
	In what follows we proceed as in Theorem 3.1 of  \cite{Ryan} (pages $374$ and $375$).  We consider two situations according to the sign of $b(x_{0})$.
	
	Assume that $b(x_{0})<0$. In this case $b(x)<0$, $\forall x\in\Sigma_{t}$. Assume that $\lambda_{1}\neq\lambda_{2}$ at some point $x\in \Sigma_{t}$, that is $p_{1}\geq1$ and $p_{2}\geq1$. Since $b(x)<0$, $\lambda_{1}$ and $\lambda_{2}$ have the same sign. As a consequence of identity $(\ref{interm1})$ and equation $(\ref{eq1'})$, we have $\lambda_{1}+\lambda_{2}=nH=p_{1}\lambda_{1}+p_{2}\lambda_{2}$, and then $(p_{1}-1)\lambda_{1}+(p_{2}-1)\lambda_{2}=0$,	which implies $p_{1}=1$ and $p_{2}=1$. In this case $n-1=p_{1}+p_{2}=2$, what is a contradiction. Hence, the multiplicity of $\lambda_{1}$ on $\Sigma_{t}$ is $n-1$. Using equation $(\ref{eq1'})$ we obtain $\lambda_{1}^2=-|T|^2 b$. This shows that $\lambda_{1}$ is constant on $\Sigma_{t}$.
	
	Assume that $b(x_{0})>0$. In this case $b(x)>0$, $\forall x\in \Sigma_{t}$. Assume that the multiplicity of $\lambda_{1}$ is $n-1$ at some point $x\in \Sigma_{t}$. Proceeding as before we conclude $0\leq\lambda_{1}^2=-|T|^2b<0$, what is a contradiction. Then $\lambda_{1}\neq\lambda_{2}$ for all $x\in \Sigma_{t}$. In this case equation $(\ref{eq1'})$ implies that
	\begin{align}\label{intsyst}
	\begin{split}
	&\lambda_{1}\lambda_{2}=-(n-2)|T|^2b,\\
	&(p_{1}-1)\lambda_{1}+(p_{2}-1)\lambda_{2}=0,
	\end{split}
	\end{align}
	and Proposition 2.2 of \cite{Ryan} allows us to conclude that the multiplicities of $\lambda_{1}$ and $\lambda_{2}$ are constant on $\Sigma_{t}$, since they are the eigenvalues of the restriction of $A$ to $\Sigma_{t}$. Assume that $p_{1}=1$. Then $p_{2}=n-2$ and, using $(\ref{intsyst})$, we get $\lambda_{2}=0$ and $(n-2)|T|^2 b=0$, which is a contradiction. The same argument applies if $p_{2}=1$. Consequently $p_{1}\geq2$ and $p_{2}\geq2$, what gives $n\geq5$. Solving $(\ref{intsyst})$ we get
	\begin{align*}
			\lambda_{i}^2&=\frac{(p_{j}-1)(n-2)}{p_{i}-1}|T|^2b,
	\end{align*}
	with $i\in\{1,2\}$ and $j=3-i$. This shows that $\lambda_{i}$ is constant on $\Sigma_{t}$.
\end{proof}

\subsection{Local Intrinsic Structure of $M^{n}$}

In this subsection we prove that an Einstein hypersurface of $I\times_{f}\mathbb{Q}^n(c)$ is locally isometric to a multiply warped product whose number of fibers is either $1$ or $2$, according to $M$ having either $2$ or $3$ principal curvatures, respectively.

Assume that $U\subset M^{n}$ is an open set for which $T$ and $b$ do not vanish, and that $A$ has three distinct eigenvalues $\lambda_{n}$, $\lambda_{1}$ and $\lambda_{2}$.

In what follows we consider the distributions $D_{i}$, $i\in\{1,2,n\}$, which assign to each $x\in U$ the vector spaces
\begin{align*}
D_{i}(x)&=\{v\in T_{x}M;(A(x)-\lambda_{i}(x)I)v=0\}.
\end{align*}
Note that the distributions defined above are mutually orthogonal. Note also that $D_{n}$ is involutive, since it is smooth and has constant rank 1. The involutivity of $D_{i}$, $i\in\{1,2\}$, is proved below.

\begin{lemma}\label{integrability}
	Assume that in the open set $U$ the shape operator $A$ has three distinct eigenvalues $\lambda_{n}$, $\lambda_{1}$ and $\lambda_{2}$ and that $T$ and $b$ do not vanish. The distributions $D_{i}$, $i\in\{1,2,n\}$, are smooth, involutive and have rank $p_{1}$, $p_{2}$ and $1$, respectively. Consequently, if $x_{0}\in U$ and $J$, $N_{1}^{p_{1}}$ and $N_{2}^{p_{2}}$ are the integral manifolds of $D_{n}$, $D_{1}$ and $D_{2}$ through $x_{0}$, respectively, then, shrinking $U$, if necessary, it is diffeomorphic to $J\times N_{1}^{p_{1}}\times N_{2}^{p_{2}}$.
\end{lemma}
\begin{proof}
	First, recall that it follows from Proposition \ref{impTeigen} and Proposition \ref{Tdiag} that $\lambda_{n}$ is smooth and has multiplicity $1$. Arguing in a similar way to {\cite[Proposition 2.2]{Ryan}, it follows that $\lambda_{1}$ and $\lambda_{2}$ have constant multiplicities on $U$, say, $p_{1}$ and $p_{2}$, respectively, and are smooth functions in this set.}
	
	Since for each $x\in U$ the distributions $D_{1}$, $\ D_{2}$ and $D_{n}$ are the eigenspaces of $A$ associated to $\lambda_{1}$, $\lambda_{2}$ and $\lambda_{n}$, respectively, they have constant ranks $p_{1}$, $p_{2}$ and $1$, respectively, in $U$. These distributions are smooth, once $\lambda_{1}$, $\lambda_{2}$ and $\lambda_{n}$ are. To see this, let $x\in U$ and consider differentiable vector fields $E_{1},\ldots,E_{n}$ around $x$, linearly independent at $x$, so that $E_{1}(x)\in D_{n}(x)$, $E_{2}(x),\ldots,E_{p_{1}+1}(x)\in D_{1}(x)$ and $E_{p_{1}+2}(x),\ldots,E_{n}(x)\in D_{2}(x)$. Define the smooth vector fields 
	\begin{align*}
	\begin{split}
	&F_{1}=(A-\lambda_{1} I)\circ(A-\lambda_{2} I)E_{1},\\
	&F_{i}=(A-\lambda_{2} I)\circ(A-\lambda_{n} I)E_{i},\ i\in\{2,\ldots,p_{1}+1\},\\
	&F_{i}=(A-\lambda_{1} I)\circ(A-\lambda_{n} I)E_{i},\ i\in\{p_{1}+2,\ldots,n\}.
	\end{split}
	\end{align*}
	We claim that $\{F_{1}\}$, $\{F_{2},\ldots,F_{p_{1}+1}\}$ and $\{F_{p_{1}+2},\ldots,F_{n}\}$ generate $D_{n}$, $D_{1}$ and $D_{2}$, respectively, around $x$, from what follows that the distributions are smooth.	The claim follows from $(A-\lambda_{i}I)\circ(A-\lambda_{j} I)(TM)\subset D_{k}$, where $i,j,k$ are distinct indices in $\{1,2,n\}$. These inclusions, in turn, follow from
	\begin{equation}\label{inclus}
	(A-\lambda_{1} I)\circ(A-\lambda_{2} I)\circ(A-\lambda_{n} I)=0,
	\end{equation}
	which reflects the fact that $\lambda_{n},\ \lambda_{1}$ and $\lambda_{2}$ are the only eigenvalues of $A$.
	
	Now, to see that $D_{i}$, $i\in\{1,2\}$, is involutive, consider $X,Y\in D_{i}$. It follows from \ref{Codeq} that $(\nabla_{X}A)Y=(\nabla_{Y}A)X$. Consequently,
	\begin{align*}
	(A-\lambda_{i}I)[X,Y]=&A(\nabla_{X}Y)-A(\nabla_{Y}X)-\lambda_{i}\nabla_{X}Y+\lambda_{i}\nabla_{Y}X\\
	=&\nabla_{X}(AY)-(\nabla_{X}A)Y-\nabla_{Y}(AX)+(\nabla_{Y}A)X\\
	&-\lambda_{i}\nabla_{X}Y+\lambda_{i}\nabla_{Y}X\\
	=&\lambda_{i}\nabla_{X}Y+X(\lambda_{i})Y-\lambda_{i}\nabla_{Y}X-Y(\lambda_{i})X\\
	&-\lambda_{i}\nabla_{X}Y+\lambda_{i}\nabla_{Y}X\\
	=&0,
	\end{align*}
	where we have used Proposition \ref{lambdannot0} to assure that $X(\lambda_{i})=Y(\lambda_{i})=0$. This implies that $[X,Y]\in D_{i}$, and then $D_{i}$ is involutive.
	
	The lemma follows from a well known application of Frobenius Theorem. See \cite{kobnom} for further details.
\end{proof}

Lemma \ref{integrability} gives a topological decomposition of $M$ on $U$. Now we seek for a compatible decomposition for the metric $g$ on $U$. The first step in this task is given in the following lemma.

\begin{lemma}\label{propos2}
	Let $X_{i}\in D_{i}$, $i\in\{1,2\}$, and $T\in D_{n}$. Then,
	\begin{enumerate}
		\item \label{item1} $\displaystyle\nabla_{T}X_{i}=[T,X_{i}]+\left(\frac{f'}{f}+\theta\lambda_{i}\right)X_{i}$;
		\item \label{item2} $\nabla_{X_{j}}X_{i}\in D_{i}$, with $i,j\in\{1,2\}$ and $i\neq j$.
	\end{enumerate}
\end{lemma}
\begin{proof}
	Item $(\ref{item1})$ follows immediately from \ref{prehess}. To prove item $(\ref{item2})$, observe that from \ref{Codeq}, $(\nabla_{X_{i}}A)X_{j}=(\nabla_{X_{j}}A)X_{i}$, and from Proposition \ref{lambdannot0}, $\nabla_{X_{2}}(\lambda_{1}X_{1})=\lambda_{1}\nabla_{X_{2}}X_{1}$ and $\nabla_{X_{1}}(\lambda_{2}X_{2})=\lambda_{2}\nabla_{X_{1}}X_{2}$. Using these identities we have
	\begin{align*}
	(A-\lambda_{2}I)(\nabla_{X_{1}}X_{2})=&\nabla_{X_{1}}(AX_{2})-(\nabla_{X_{1}}A)X_{2}-\lambda_{2}\nabla_{X_{1}}X_{2}\\
	=&\nabla_{X_{1}}(\lambda_{2}X_{2})-(\nabla_{X_{1}}A)X_{2}-\lambda_{2}\nabla_{X_{1}}X_{2}\\
	=&\nabla_{X_{2}}(\lambda_{1}X_{1})-(\nabla_{X_{2}}A)X_{1}-\lambda_{1}\nabla_{X_{2}}X_{1}\\
	=&\nabla_{X_{2}}(AX_{1})-(\nabla_{X_{2}}A)X_{1}-\lambda_{1}\nabla_{X_{2}}X_{1}\\
	=&(A-\lambda_{1}I)(\nabla_{X_{2}}X_{1}).
	\end{align*}
	Observe that from $(\ref{inclus})$, we get $(A-\lambda_{i}I)(TM)\subset D_{j}\oplus D_{k}$, $i,j,k$ all distinct in $\{1,2,n\}$. This implies that $(A-\lambda_{1}I)(\nabla_{X_{2}}X_{1})=(A-\lambda_{2}I)(\nabla_{X_{1}}X_{2})\in D_{n}$. But since,
	\begin{align*}
	\left\langle (A-\lambda_{1}I)\nabla_{X_{2}}X_{1},T\right\rangle&=\left\langle\nabla_{X_{2}}X_{1},AT\right\rangle-\lambda_{1}\left\langle\nabla_{X_{2}}X_{1},T\right\rangle\\
	&=(\lambda_{n}-\lambda_{1})\left\langle\nabla_{X_{2}}X_{1},T\right\rangle\\
	&=-(\lambda_{n}-\lambda_{1})\left\langle X_{1},\nabla_{X_{2}}T\right\rangle\\
	&=-(\lambda_{n}-\lambda_{1})\left(\frac{f'}{f}\left\langle X_{1},X_{2}\right\rangle+\theta \left\langle X_{1},AX_{2}\right\rangle\right)\\
	&=-(\lambda_{n}-\lambda_{1})\theta\lambda_{2}\left\langle X_{1},X_{2}\right\rangle\\
	&=0,
	\end{align*}
	we have $(A-\lambda_{1}I)(\nabla_{X_{2}}X_{1})=(A-\lambda_{2}I)(\nabla_{X_{1}}X_{2})=0$, and then $\nabla_{X_{2}}X_{1}\in D_{1}$ and $\nabla_{X_{1}}X_{2}\in D_{2}$, finishing the proof.
\end{proof}

If $J$ is an integral curve of $T$ with arc length parameter $s$, one has $T=|T|\partial_{s}$ and, from Propositions \ref{impTeigen} and \ref{lambdannot0}, $\lambda_{1},\ \lambda_{2},\ |T|$ and $\theta$ are functions only of $s$ in $U$. On the other hand, the restrictions of $f$ and $f'$ to $M$ obey the relation $\frac{df}{ds}=|T|f'$. With this in mind, the decomposition of $g$ in $U$ is given in the next result.
\begin{proposition}\label{prop1}
	There are Riemannian metrics $g_{1}$ and $g_{2}$ on $N_{1}$ and $N_{2}$, the manifolds of Lemma \ref{integrability}, respectively, so that the open set $U$ with the Riemannian metric induced by $M$ is locally isometric to $J\times N_{1}^{p_{1}}\times N_{2}^{p_{2}}$ with the metric
	\begin{align}\label{doubwarp}
		g=ds^2+\varphi_{1}(s)^2g_{1}+\varphi_{2}(s)^2g_{2},
	\end{align}
	where
	\begin{align}\label{labelfi}
	\frac{1}{\varphi_{i}}\frac{d\varphi_{i}}{ds}=\frac{1}{|T|^2f}\frac{df}{ds}+\frac{\theta\lambda_{i}}{|T|},\ i\in\{1,2\}.
	\end{align}
	The manifold $J\times N_{1}^{p_{1}}\times N_{2}^{p_{2}}$ regarded with the metric $(\ref{doubwarp})$ is called multiply warped product and it is denoted by $J\times_{\varphi_{1}}N_{1}\times_{\varphi_{2}}N_{2}$.
\end{proposition}
\begin{proof}
	Let $x_{0}\in U$ and consider a coordinate system $(s,x_{1},\ldots,x_{p_{1}},y_{1},\ldots,y_{p_{2}})$ around $x_{0}$ for which $x_{0}=(0,\ldots,0)$, where  $s\in J$ is the arc length parameter, $(x_{1},\ldots,x_{p_{1}})\in N_{1}^{p_{1}}$ and $(y_{1},\ldots,y_{p_{2}})\in N_{2}^{p_{2}}$. Consequently, $\partial_{s}\in D_{n}$, $\partial_{x_{1}},\ldots,\partial_{x_{p_{1}}}\in D_{1}$ and $\partial_{y_{1}},\ldots,\partial_{y_{p_{2}}}\in D_{2}$. Let $g_{ij}$ be the representation of the metric of $M$ in these coordinates. If $i,j\in\{1,\ldots,p_{1}\}$ and $k,l\in\{1,\ldots,p_{2}\}$, then item $(\ref{item2})$ of Lemma \ref{propos2} gives $\nabla_{\partial_{x_{i}}}\partial_{y_{k}}=\nabla_{\partial_{y_{k}}}\partial_{x_{i}}\in D_{1}\cap D_{2}=\{0\}$, and then
	\begin{align*}
	\partial_{y_{k}}(g_{ij})=&\left\langle\nabla_{\partial_{y_{k}}}\partial_{x_{i}},\partial_{x_{j}}\right\rangle+\left\langle\partial_{x_{i}},\nabla_{\partial_{y_{k}}}\partial_{x_{j}}\right\rangle=0,\\
	\partial_{x_{i}}(g_{kl})=&\left\langle\nabla_{\partial_{x_{i}}}\partial_{y_{k}},\partial_{y_{l}}\right\rangle+\left\langle\partial_{y_{k}},\nabla_{\partial_{x_{i}}}\partial_{y_{l}}\right\rangle=0
	\end{align*}
	Thus, $g_{ij}$ restricted to $N_{1}^{p_{1}}$ does not depend on the coordinates $(y_{1},\ldots,y_{p_{2}})$, and $g_{kl}$ restricted to $N_{2}^{p_{2}}$ does not depend on the coordinates $(x_{1},\ldots,x_{p_{1}})$. On the other hand, using item $(\ref{item1})$ of Lemma \ref{propos2} and $[\partial_{x_{i}},\partial_{s}]=0$ we obtain
	\begin{align*}
	\partial_{s}(g_{ij})=&\left\langle\nabla_{\partial_{s}}\partial_{x_{i}},\partial_{x_{j}}\right\rangle+\left\langle\partial_{x_{i}},\nabla_{\partial_{s}}\partial_{x_{j}}\right\rangle\\
	=&\frac{2}{|T|}\left(\frac{f'}{f}+\theta\lambda_{1}\right)g_{ij},
	\end{align*}
	which gives
	\begin{align*}
	\partial_{s}(\ln|g_{ij}|)=2\left(\frac{1}{|T|^2f}\frac{df}{ds}+\frac{\theta\lambda_{1}}{|T|}\right).
	\end{align*}
	Analogously,
	\begin{align*}
	\partial_{s}(\ln|g_{kl}|)=2\left(\frac{1}{|T|^2f}\frac{df}{ds}+\frac{\theta\lambda_{2}}{|T|}\right).
	\end{align*}
	Note that $\partial_{s}(\ln|g_{ij}|)$ and $\partial_{s}(\ln|g_{kl}|)$ depend only on $s$. Integrating them from $0$ to $s\in J$ and using the expressions above we get
	\begin{align}\label{dcpst1}
	\begin{split}
	g_{ij}(s,x_{1},\ldots,x_{p_{1}},y_{1},\ldots,y_{p_{2}})&=\varphi_{1}(s)^2g_{ij}(0,x_{1},\ldots,x_{p_{1}},0,\ldots,0),\\
	g_{kl}(s,x_{1},\ldots,x_{p_{1}},y_{1},\ldots,y_{p_{2}})&=\varphi_{2}(s)^2g_{kl}(0,0,\ldots,0,y_{1},\ldots,y_{p_{2}}),
	\end{split}
	\end{align}
	where $\varphi_{i}$ is given by $(\ref{labelfi})$. Now, consider the Riemannian metrics
	\begin{align}\label{notation}
	\begin{split}
	(g_{1})_{ij}(x_{1},\ldots,x_{p_{1}})=g_{ij}(0,x_{1},\ldots,x_{p_{1}},0,\ldots,0)\\
	(g_{2})_{kl}(y_{1},\ldots,y_{p_{2}})=g_{kl}(0,0,\ldots,0,y_{1},\ldots,y_{p_{2}}),
	\end{split}
	\end{align}
	on $N_{1}^{p_{1}}$ and $N_{2}^{p_{2}}$, respectively. Using $(\ref{dcpst1})$, $(\ref{notation})$, $g(\partial_{s},\partial_{s})=1$, $g(\partial_{s},\partial_{y_{k}})=0$, $g(\partial_{s},\partial_{x_{i}})=0$ and $g(\partial_{x_{i}},\partial_{y_{k}})=0$, we obtain $(\ref{doubwarp})$, finishing the proof.
\end{proof}

If $U\subset M^{n}$ is an open set where $A$ has exactly two distinct eigenvalues at each $x\in U$, say, $\lambda_{n}$ and $\lambda$, applying similar arguments one has
\begin{proposition}\label{prop2}
	There is a Riemannian metric $g_{N}$ on $N^{n-1}$ and so that the open set $U$ with the Riemannian metric induced by $M$ is locally isometric to $J\times N^{n-1}$ with the metric	\begin{align*}
	g=ds^2+\varphi(s)^2g_{N},
	\end{align*}
	where
	\begin{align*}
		\frac{1}{\varphi}\frac{d\varphi}{ds}=\frac{1}{|T|^2f}\frac{df}{ds}+\frac{\theta\lambda}{|T|}.
	\end{align*}
\end{proposition}

In the next section we apply Proposition \ref{prop1} and Proposition \ref{prop2} to prove Theorem \ref{intM}, Theorem \ref{rigidity1} and Theorem \ref{thmlambda_{n}=0}.

\section{Applications}

In this section we apply the results obtained in the last section. An immediate consequence is the following one.
\begin{proof}[{\bf Proof of Theorem \ref{intM}}]
	Follows essentially from Proposition \ref{Tdiag} and Proposition \ref{prop1}.
\end{proof}

Next, we prove Theorem \ref{rigidity1}, Theorem \ref{thmlambda_{n}=0} and use our arguments to prove the main theorem of \cite{SantosPN}.

\begin{proof}[{\bf Proof of Theorem \ref{rigidity1}}]
By Proposition \ref{Tdiag} and Proposition \ref{prop2} it follows that $U$ can be taken isometric to $J\times_{\varphi}N^{n-1}$. Now consider orthogonal vector fields $X,Y\in\mathfrak{X}(N)$. Using the formulas for the sectional curvature of a multiply warped product (which may be found for example in \cite{wpmandlcfs}, formulas $(3)$, where their curvature tensor differs from ours by a sign) one gets
\begin{align*}
	&\left(sec_{N}(X,Y)-\left(\frac{d\varphi}{ds}\right)^2\right)\frac{1}{\varphi^2}=sec_{M}(X,Y)=\lambda^2-a,
\end{align*}
where in the right hand side we have used the Gauss formula or, more precisely, equation \ref{Gausseq}. Rewriting the equation above we get
\begin{align*}
	&\left(\frac{d\varphi}{ds}\right)^2+(\lambda^2-a)\varphi^2=sec_{N}(X,Y).
\end{align*}
Now we observe that from the equation above, there is a constant $k$ such that $sec_{N}(X,Y)=k$, for any orthonormal vector fields $X,Y\in\mathfrak{X}(N)$. As a consequence, the warped product $J\times_{\varphi}N^{n-1}$ is locally conformally flat \cite{wpmandlcfs,brozosvazquez,yau}, and since it is isometric to $U$, which is Einstein, $M$ has constant sectional curvature on $U$.
\end{proof}

When $M$ has three distinct principal curvatures on $U$, Proposition \ref{prop1} says that it is a multiply warped product with one dimensional base and two fibers. The curvature of such spaces can be expressed in terms of the warping functions (i.e., $\varphi_{1}$ and $\varphi_{2}$) and the curvatures of the fibers. Such expressions can be found, for example, in \cite{wpmandlcfs}. This and Gauss formula yield the following result.
\begin{proposition}\label{step0}
	Let $(M^{n},g)$, $n\geq4$, be an Einstein hypersurface of $I\times_{f}\mathbb{Q}^{n}(c)$ and suppose that $T$ and $b$ do not vanish on an open set $U\subset M$. If $M$ has exactly three distinct principal curvatures on $U$, say, $\lambda_{1},\ \lambda_{2}$ and $\lambda_{n}$, with multiplicities $p_{1},\ p_{2}$ and $1$, respectively, then $U$ is locally isometric to $J\times_{\varphi_{1}}N_{1}^{p_{1}}\times_{\varphi_{2}}N_{2}^{p_{2}}$, where $(N_{i}^{p_{i}},g_{N_{i}})$ is a manifold of constant sectional curvature provided $p_{i}>1$. Furthermore, the following equations are satisfied
\begin{align}
	&\frac{d^{2}\varphi_{i}}{ds^{2}}=(a+b|T|^2-\lambda_{i}\lambda_{n})\varphi_{i}\label{id1}\\
	&\left(\frac{d\varphi_{i}}{ds}\right)^2+(\lambda_{i}^2-a)\varphi_{i}^2=k_{i}\label{id2},\ \text{if}\  p_{i}>1\\
	&\frac{d\varphi_{1}}{ds}\frac{d\varphi_{2}}{ds}=(a-\lambda_{1}\lambda_{2})\varphi_{1}\varphi_{2}\label{id3}
\end{align}
where $i\in\{1,2\}$.
\end{proposition}
\begin{proof}
	The idea is to use the curvature formulas of a multiply warped product $J\times_{\varphi_{1}}N_{1}^{p_{1}}\times_{\varphi_{2}}N_{2}^{p_{2}}$ (which can be found in \cite{brozosvazquez}) together with Gauss equation (equation \ref{Gausseq}). The existence of the constant $k_{i}$ follows in a similar fashion to the proof of Theorem \ref{rigidity1}.
\end{proof}

Equations $(\ref{id1})$-$(\ref{id3})$ seem to play an important role in the investigation of Einstein hypersurfaces of $I\times_{f}\mathbb{Q}^n(c)$ with $3$ distinct principal curvatures. One application is in the proof of Theorem \ref{thmlambda_{n}=0}, which deals with the case where $\lambda_{n}$ vanishes.

\begin{proof}[{\bf Proof of Theorem \ref{thmlambda_{n}=0}}]
	Since $M$ does not have constant sectional curvature, it follows from Theorem \ref{intM} and Theorem \ref{rigidity1} that $M$ has three distinct principal curvatures. Now, Proposition \ref{lambdannot0} and the assumption $\lambda_{n}\equiv0$ assures that $p_{1}\geq2$ and $p_{2}\geq2$, and equations $(\ref{eq1})$ and $(\ref{eq2})$ imply that
	\begin{align}\label{threeEQ}
		\begin{split}
			&a+|T|^2 b=-\frac{\rho}{n-1},\\
			&\lambda_{1}\lambda_{2}=-(n-2)|T|^2b<0,\\
			&(p_{1}-1)\lambda_{1}+(p_{2}-1)\lambda_{2}=0.
		\end{split}
	\end{align}
	Since $\lambda_{1}\lambda_{2}<0$, we will assume that $\lambda_{1}<0<\lambda_{2}$. Manipulating $(\ref{threeEQ})$ one gets,
	\begin{align*}
		a=-\frac{\rho}{n-1}+\frac{1}{n-2}\lambda_{1}\lambda_{2}
	\end{align*}
	and then
	\begin{align*}
		\lambda_{i}^{2}-a=\frac{(n-3)p_{j}}{(n-2)(p_{j}-1)}\lambda_{i}^{2}+\frac{\rho}{n-1}.
	\end{align*}
	Fix $i\in\{1,2\}$ and let $j=3-i$. Inserting the relations above into $(\ref{id1})$, $(\ref{id2})$ and $(\ref{id3})$ we obtain
	\begin{align}
		&\frac{d^{2}\varphi_{i}}{ds^{2}}=-\frac{\rho}{n-1}\varphi_{i},\label{id1'}\\
		&\left(\frac{d\varphi_{i}}{ds}\right)^2+\left(\frac{(n-3)p_{j}}{(n-2)(p_{j}-1)}\lambda_{i}^{2}+\frac{\rho}{n-1}\right)\varphi_{i}^2=k_{i},\label{id2'}\\
		&\frac{d\varphi_{1}}{ds}\frac{d\varphi_{2}}{ds}=-\left(\frac{\rho}{n-1}+\frac{n-3}{n-2}\lambda_{1}\lambda_{2}\right)\varphi_{1}\varphi_{2}.\label{id3'}
	\end{align}
	Using $(\ref{id1'})$ and $(\ref{id2'})$, we obtain
	\begin{align*}
		\frac{(n-3)p_{j}}{(n-2)(p_{j}-1)}\frac{d}{ds}\left(\lambda_{i}^{2}\varphi_{i}^{2}\right)=\frac{d}{ds}\left(k_{i}-\left(\frac{d\varphi_{i}}{ds}\right)^2-\frac{\rho}{n-1}\varphi_{i}^{2}\right)=0.
	\end{align*}
	Consider positive constants $A_{1}$ and $A_{2}$ satisfying
	\begin{align}\label{linrela}
		\varphi_{1}=-\frac{A_{1}}{\lambda_{1}}\ \ \ \text{and}\ \ \ \varphi_{2}=\frac{A_{2}}{\lambda_{2}}.
	\end{align}
	In view of the last equation of $(\ref{threeEQ})$ we get
	\begin{align}\label{multiple}
		\varphi_{2}=\frac{(p_{2}-1)A_{2}}{(p_{1}-1)A_{1}}\varphi_{1}
	\end{align}
	Using this equality in $(\ref{labelfi})$ one has
	\begin{align*}
		\frac{1}{|T|^2f}\frac{df}{ds}+\frac{\theta\lambda_{1}}{|T|}=\frac{1}{\varphi_{1}}\frac{d\varphi_{1}}{ds}=\frac{1}{\varphi_{2}}\frac{d\varphi_{2}}{ds}=\frac{1}{|T|^2f}\frac{df}{ds}+\frac{\theta\lambda_{2}}{|T|}
	\end{align*}
	and since $\lambda_{1}\neq\lambda_{2}$, one obtains $\theta\equiv0$. Therefore, $|T|^2\equiv1$, $s=t$ and then 
	\begin{align}\label{goodrelation}
		\frac{1}{f}\frac{df}{dt}=\frac{1}{\varphi_{1}}\frac{d\varphi_{1}}{dt}=\frac{1}{\varphi_{2}}\frac{d\varphi_{2}}{dt},
	\end{align}
	proving the existence of a constant $B_{i}>0$ so that $\varphi_{i}=B_{i}f$. On the other hand, using the last equation of $(\ref{threeEQ})$, $(\ref{linrela})$ and $(\ref{multiple})$ in $(\ref{id3'})$ we get
	\begin{align}\label{stp1}
			B_{i}^2\left(\left(\frac{df}{dt}\right)^2+\frac{\rho}{n-1}f^2\right)=\left(\frac{d\varphi_{i}}{dt}\right)^2+\frac{\rho}{n-1}\varphi^2_i=\frac{(p_{i}-1)(n-3)A_{i}^2}{(p_{j}-1)(n-2)},
	\end{align}
	and  using the last equation of $(\ref{id2'})$ and $(\ref{linrela})$ we obtain
	\begin{align}\label{stp2}
		\left(\frac{d\varphi_{i}}{dt}\right)^2+\frac{\rho}{n-1}\varphi^2_i=k_{i}-\frac{(n-3)p_{j}A_{i}^{2}}{(n-2)(p_{j}-1)}.
	\end{align}
	From $(\ref{stp1})$ and $(\ref{stp2})$ we conclude that
	\begin{align}\label{rela1}
		\frac{(p_{1}-1)A_{1}}{B_{1}}=\frac{(p_{2}-1)A_{2}}{B_{2}},\ \ \ \ k_{i}=\frac{(n-3)A_{i}^2}{p_{j}-1}>0.
	\end{align}
	On the other hand, by $(\ref{defint.a})$, $(\ref{id3})$ and $(\ref{goodrelation})$ we have
	\begin{align*}
		\frac{(f')^2-c}{f^2}=\frac{1}{\varphi_{1}\varphi_{2}}\frac{d\varphi_{2}}{dt}\frac{d\varphi_{2}}{dt}+\lambda_{1}\lambda_{2}=\frac{(f')^2}{f^2}+\lambda_{1}\lambda_{2}
	\end{align*}
	and from $(\ref{linrela})$ and $\varphi_{i}=B_{i}f$,
	\begin{align*}
		-\frac{c}{f^2}=\lambda_{1}\lambda_{2}=-\frac{A_{1}A_{2}}{B_{1}B_{2}f^2}<0,
	\end{align*}
	from where it follows that $c=1$ and $A_{1}A_{2}=B_{1}B_{2}$. This together with $(\ref{rela1})$ implies that the constants $A_{i}$ and $B_{i}$ are given by
	\begin{align*}
		A_{i}=\sqrt{\frac{p_{j}-1}{p_{i}-1}}B_{i},\ \ \ \ \ \ B_{i}=\sqrt{\frac{(p_{i}-1)k_{i}}{n-3}}.
	\end{align*}
	Using these expressions in $\varphi_{i}=B_{i}f$, $(\ref{linrela})$ and $(\ref{stp1})$ we prove the necessity. The proof that $(\ref{example1})$ is immersed in $I\times_{f}\mathbb{S}^n$ is a simple computation, where one has to check all conditions of Theorem \ref{tensorR}.
\end{proof}

The last result of this paper concerns cylinders, that is, the case where $f$ constant.

\begin{theorem}[\cite{SantosPN}]\label{rigidity2}
	Let $(M^{n},g)$, $n\geq4$, be an Einstein manifold locally immersed in $I\times_{f}\mathbb{Q}^{n}(c)$ and suppose that $T$ does not vanish on an open set $U\subset M$. If $f$ is constant and $c\neq0$, then $M$ has constant sectional curvature on $U$.
\end{theorem}
\begin{proof}
	Assume by contradiction that $M$ does not have constant sectional curvature on $U$ and that $f$ is constant. Without loss of generality, assume that $f\equiv1$.
	
	Theorem \ref{intM} and Theorem \ref{rigidity1} assure that $M$ has three distinct principal curvatures on a open set $U\subset M$. On the other hand, if $\lambda_{n}$ vanishes, $U$ is locally isometric to $(\ref{example1})$, by Theorem \ref{thmlambda_{n}=0}. But since $f$ is constant, it follows from $(\ref{id1'})$ that its Einstein constant is $0$, what is a contradiction. Thus, $\lambda_{n}$ does not vanish. Consequently, it follows from item $(\ref{gradtheta})$ of Proposition \ref{impTeigen} that $\theta$ is not constant.
	
	On the other hand, from $f\equiv1$, Proposition \ref{prop1} and Proposition \ref{step0} we have
	\begin{align*}
		\frac{d\varphi_{i}}{ds}=\frac{\theta\lambda_{i}}{|T|}\varphi_{i} ,\ \ \ \ \ \ \frac{d\varphi_{1}}{ds}\frac{d\varphi_{2}}{ds}=-(c+\lambda_{1}\lambda_{2})\varphi_{1}\varphi_{2},
	\end{align*}
	and then $\lambda_{1}\lambda_{2}=-c|T|^2$. Since we already have $\lambda_{1}\lambda_{2}=\rho-(n-1)c+c|T|^2$, which follows from $(\ref{eq1})$, we get $2c|T|^2=(n-1)c-\rho$, which implies that $\theta$ is constant, and this is a contradiction. Thus, $M$ has constant sectional curvature.
	\end{proof}

\end{document}